\documentclass[11pt]{amsart}
\usepackage[dvips]{graphicx,color}
\usepackage{xcolor}
\usepackage{amsthm,amssymb,amscd,latexsym,epsfig,xypic,hyperref}
\usepackage[mathscr]{euscript}
\DeclareRobustCommand{\SkipTocEntry}[3]{}
\makeatletter
\newcommand\@dotsep{4.5}
\def\@tocline#1#2#3#4#5#6#7{\relax
  \ifnum #1>\c@tocdepth 
  \else
    \par \addpenalty\@secpenalty\addvspace{#2}%
    \begingroup \hyphenpenalty\@M
    \@ifempty{#4}{%
      \@tempdima\csname r@tocindent\number#1\endcsname\relax
    }{%
      \@tempdima#4\relax
    }%
    \parindent\z@ \leftskip#3\relax \advance\leftskip\@tempdima\relax
    \rightskip\@pnumwidth plus1em \parfillskip-\@pnumwidth
    #5\leavevmode #6\relax
    \leaders\hbox{$\m@th
      \mkern \@dotsep mu\hbox{.}\mkern \@dotsep mu$}\hfill
    \hbox to\@pnumwidth{\@tocpagenum{#7}}\par
    \nobreak
    \endgroup
  \fi}
\makeatother

\DeclareFontFamily{OT1}{rsfs}{}
\DeclareFontShape{OT1}{rsfs}{n}{it}{<-> rsfs10}{}
\DeclareMathAlphabet{\curly}{OT1}{rsfs}{n}{it}

\newcommand{\virt}{\mbox{\tiny virt}}

\newcommand{\zz}{{\mathbb Z}}

\newcommand{\kk}{{\mathbb K}}

\newcommand{\aaa}{{\mathbb A}}

\renewcommand{\ll}{{\mathbb L}}
\newcommand{\qq}{{\mathbb Q}}

\newcommand{\cc}{{\mathbb C}}

\newcommand{\Gm}{{{\mathbb G}_{\mbox{\tiny\rm m}}}}

\newcommand{\eE}{{\mathcal E}}

\newcommand{\sS}{{\mathcal S}}
\newcommand{\tS}{{\widetilde S}}
\newcommand{\tR}{{\widetilde R}}
\newcommand{\tA}{{\widetilde A}}
\newcommand{\tB}{{\widetilde B}}

\newcommand{\oO}{{\mathcal O}}

\newcommand{\wW}{{\mathcal W}}
\newcommand{\mM}{{\mathcal M}}

\newcommand{\mMF}{\mathscr{MF}}

\newcommand{\s}{{\widetilde s}}

\newcommand{\cHom}{\mathscr{H}om}
\newcommand{\rE}{\mathscr{E}}
\newcommand{\E}{\mathop{\sf E}\nolimits}

\DeclareMathOperator{\Var}{Var}

\DeclareMathOperator{\val}{val}
\DeclareMathOperator{\id}{id}
\DeclareMathOperator{\Crit}{Crit}

\DeclareMathOperator{\Hilb}{Hilb}

\DeclareMathOperator{\vd}{vd}

\DeclareMathOperator{\ind}{ind}
\DeclareMathOperator{\Sym}{Sym}
\DeclareMathOperator{\Hom}{Hom}
\DeclareMathOperator{\Spec}{Spec}

\DeclareMathOperator{\Ob}{Ob}
\DeclareMathOperator{\Tot}{Tot}
\DeclareMathOperator{\rank}{rank}
\DeclareMathOperator{\ob}{ob}
\DeclareMathOperator{\red}{red}
\DeclareMathOperator{\Exp}{Exp}
\DeclareMathOperator{\SU}{SU}
\DeclareMathOperator{\VW}{VW}
\DeclareMathOperator{\vw}{vw}

\renewcommand\O{\mathcal O}
\newcommand\LL{\mathbb L}

\newcommand\C{\mathbb C}

\newcommand\udot{^{\bullet}}
\renewcommand\;{\hspace{1pt}}

\newcommand{\Rt}[1]{\stackrel{#1\,}{\longrightarrow}}
\newcommand{\RT}[2]{\xymatrix@C=#1pt{\ar[r]^{#2}&}}
\newcommand\To{\longrightarrow}

\newcommand\ito{\ar@{^{ (}->}[r]}

\newfont{\bigtimesfont}{cmsy10 scaled \magstep5}
\newcommand{\bigtimes}{\mathop{\lower0.9ex\hbox{\bigtimesfont\symbol2}}}

\newcommand\rk{\operatorname{rank}}
\newcommand\tr{\operatorname{tr}}

\newcommand\Pic{\operatorname{Pic}}

\newcommand\beq[1]{\begin{equation}\label{#1}}
\newcommand\eeq{\end{equation}}
\newcommand\beqa{\begin{eqnarray*}}
\newcommand\eeqa{\end{eqnarray*}}

\makeatletter \@addtoreset{equation}{section} \makeatother

\newtheorem{defn}[equation]{Definition}
\newtheorem{thm}[equation]{Theorem}

\newtheorem{cor}[equation]{Corollary}
\newtheorem{prop}[equation]{Proposition}
\newenvironment{rmk}{\noindent\textbf{Remark}.}{}

\theoremstyle{definition}

\DeclareRobustCommand{\SkipTocEntry}[4]{}

\begin{document}

\title[Virtual signed Euler characteristics and VW invariants]{Motivic virtual signed Euler characteristics and applications to Vafa-Witten invariants}
\author{Yunfeng Jiang}

\begin{abstract}
%
%
For any scheme $M$ with a perfect obstruction theory, Jiang and Thomas associate a scheme $N$ with symmetric perfect obstruction theory. The scheme $N$ is a cone over $M$ given by the dual of the obstruction sheaf of $M$, and contains $M$ as its zero section. Locally $N$ is the critical locus of a regular function.
 In this  note we prove that $N$ is a $d$-critical scheme in the sense of Joyce.  By assuming an orientation on $N$ there exists a global motive for $N$ locally given by the motive of vanishing cycles of the local regular function. We prove a motivic localization formula under the good and circle compact $\C^*$-action for $N$. When taking Euler characteristic  the weighted Euler characteristic of $N$ weighted by the Behrend function  is the signed Euler characteristic of $M$ by motivic method. 
  
 As applications we calculate the motivic generating series of the motivic Vafa-Witten invariants for K3 surfaces. This motivic series gives the result of the $\chi_y$-genus for Vafa-Witten invariants of K3 surfaces, which is the same (at instanton branch) as the K-theoretical Vafa-Witten invariants of Thomas.
\end{abstract}

\maketitle \vspace{-5mm}
\tableofcontents

\section{Introduction}

Let $M$ be a scheme (or DM stack) with a perfect obstruction theory $E^\bullet$ in the sense of \cite{LT}, \cite{BF}.  Denote by 
$\vd=\rank(E^\bullet)$, which is the virtual dimension.  In \cite{JT}, Jiang and Thomas associate with $M$ a scheme $N$, which admits a symmetric obstruction theory in the sense of Behrend in \cite{Behrend}.  Roughly this scheme $N=\Tot(\ob_{M}^*)$ is the total space of the dual of the obstruction sheaf $\ob_{M}$ associated with the obstruction theory $E^\bullet$.  For the  mathematical  definition,  
$$N:=\Spec(\Sym^{\bullet}\ob_{M}),$$
which is the abelian cone of the obstruction sheaf $\ob_M$.  Let $\pi: N\to M$ be the projection. 
There is a $\Gm$-action on $N$ scaling the fibers and the fixed locus is $M$. 
In \cite{JT}, the following four invariants  are studied: 
\begin{enumerate}
\item The Ciocan-Fontanine-Kapranov/Fantechi-G\"ottsche signed virtual Euler characteristic of $M$ defined using its own obstruction theory;
\item Graber-Pandharipande's virtual Atiyah-Bott localization of the virtual cycle of $N$ to $M$;
\item Behrend's weighted Euler characteristic of $N$ by the Behrend function $\nu_N$;
\item Kiem-Li's cosection localization of the virtual cycle of $N$ to $M$. 
\end{enumerate}
\cite{JT} proves that $(1)=(2)$, and $(3)=(4)$.  The invariant $(1)=(2)$ is deformation invariant, while $(3)=(4)$ is not. 
Moreover the invariant  $(3)=(4)$ is the signed Euler characteristic of $M$. 
A similar related situation was studied in \cite{Jiang4}. 
It turns out that both of these invariants are useful, which are related to the Vafa-Witten invariants for projective surfaces or local surfaces, see \cite{TT}, \cite{TT2}. 

In the setting of derived algebraic geometry as in \cite{PTVV}, $N$ is $(-1)$-shifted symplectic because it is the $(-1)$-shifted cotangent bundle  $T^*(M, E^\bullet)[-1]$ of $(M, E^\bullet)$, if $(M, E^\bullet)$ comes from a quasi-smooth derived scheme.  In \cite{Joyce}, Joyce proves that the underlying scheme $N$ of a $(-1)$-shifted symplectic derived scheme is a $d$-critical scheme, which is defined in \cite{Joyce}, but the converse is not true.  Since a scheme $M$ with a perfect obstruction theory $E^\bullet$ is not always induced from a quasi-smooth derived scheme, in the paper \cite{JT}, the authors start from a derived scheme $(M, E^\bullet)$, and construct $N$ by taking derived cotangent bundle of $(M, E^\bullet)$.  

We prove in this paper that $N$ is a $d$-critical scheme in the sense of \cite{Joyce}.  Of course $N$ does not always come from a $(-1)$-shifted symplectic derived scheme.  
Also $N$ does not always have a symmetric obstruction theory in \cite{Behrend}.  
Let $K_{N}$ be the canonical line bundle for the $d$-critical scheme $N$ defined in
 \cite[Definition 2.31]{Joyce}. We assume that the $d$-critical scheme $N$ has an orientation, i.e., a square root $K_{N}^{\frac{1}{2}}$ exists.  So from \cite{BJM}, there is a unique global motive $\mMF_{N}^\phi\in \overline{\mM}_{N}^{\hat{\mu}}$, where $\overline{\mM}_{N}^{\hat{\mu}}=K^{\hat{\mu}}_0(\Var_{N})[\mathbb{L}^{-1}]$ and $K^{\hat{\mu}}_0(\Var_{N})$ is the 
equivariant Grothendieck ring of varieties. 
On each $d$-critical chart  $(\widetilde{R},  \widetilde{A}, \s, \widetilde{i})$ of the $d$-critical scheme $(N,s)$, the motive 
$$\mMF_{N}^{\phi}|_{\widetilde{R}}=i^\star(\mMF_{\widetilde{A},\s}^{\phi})\odot \Upsilon(Q_{\widetilde{R},  \widetilde{A}, \s, \widetilde{i}})\in \overline{\mM}_{\widetilde{R}}^{\hat{\mu}},$$
where 
$\mMF_{\widetilde{A},\s}^{\phi}=\ll^{-\dim(\widetilde{A})/2}\odot [[\widetilde{A}_0,\hat{\iota}]-\mMF_{\widetilde{A},\s}]|_{\widetilde{R}}$ is the motivic vanishing cycle; 
$\mMF_{\widetilde{A},\s}$ is the motivic nearby cycle defined in \cite[Definition 7.9]{Jiang3}; and 
$\Upsilon(Q_{\widetilde{R},  \widetilde{A}, \s, \widetilde{i}})=\ll^{\frac{1}{2}}\odot ([\widetilde{R},\hat{\iota}]-[Q,\hat{\rho}])\in \overline{\mM}_{\widetilde{R}}^{\hat{\mu}}$ is the motive of the principal $\zz_2$-bundle $Q_{\widetilde{R},  \widetilde{A}, \s, \widetilde{i}}$ as in \cite[\S 2.5]{BJM}.
This ring admits a new product $\odot$, which is defined by \cite{BJM}.  More details can be found in \cite[\S 7.1]{Jiang3}. 

In this paper we prove a motivic localization formula for the global motive $\mMF_{N}^\phi$ under the $\Gm$-action on $N$, removing the preservation of the orientation $K_{N}^{\frac{1}{2}}$ of the $\Gm$-action as in \cite[Theorem 7.17]{Jiang3}.   

The cone $N$ admits a  good, circle-compact action of $\Gm$ by scaling the fibers of $N$. The fixed locus is $M$.  Our main result is: 

\begin{thm}\label{intro_thm_general_motivic_localization}(Theorem \ref{thm_general_motivic_localization})
Let $(N,s)$ be the oriented $d$-critical scheme in Proposition \ref{prop_d_critical_scheme_N} and $\mu$ is the good, circle-compact action of  $\Gm$ on $N$. Let  
$\mMF_{N}^{\phi}\in\overline{\mM}_{N}^{\hat{\mu}}$ be the global motive of $N$.   Then we have the following localization formula. 
$$\int_{N}\mMF_{N}^{\phi}=\ll^{-\frac{\rank(E^{\bullet})}{2}}\odot \int_{M}[M].$$ 
\end{thm}

Our method to prove Theorem \ref{intro_thm_general_motivic_localization} follows from the argument of the motivic nearby cycle under the good, circle-compact action of $\Gm$ in \cite{NS}.  In our situation Nicaise and Payne \cite{NS} prove the conjecture of 
Davison and Meinhardt on the motivic nearby fiber in \cite{DM17}, i.e., the motivic nearby cycle $\mMF_{\widetilde{A},\s}=[\s^{-1}(1)]$. 
We use this motivic cycle to prove the localization formula.   This is parallel to the result in Theorem A.1 of \cite{Dave}, where the author proves that in this situation
the pushforward to $M$ of the vanishing cycle sheaf $\phi_{\s}$ of the local regular function $\s$ on a $d$-critical chart
is the shifted constant sheaf on $M$.  Our contribution here is that we prove this statement for the motivic vanishing cycle using the result by Nicaise and Payne in \cite{NS} who use the techniques of semi-algebraic subsets.  Note that we do not need to use the techniques of semi-algebraic sets in \cite{NS}, only a result there. 
As an application of Theorem  \ref{intro_thm_general_motivic_localization}, we show that 
$$\chi(N, \nu_N)=(-1)^{\vd}\chi(M)$$
by the motivic method.  This equality was proved by a direct topological calculation in \cite{JT}.   

\medskip
\noindent\textbf{Motivic Vafa-Witten theory.} 
Let $S$ be a smooth projective surface. Motived by S-duality and Vafa-Witten's equation in \cite{VW},  Tanaka-Thomas \cite{TT}, \cite{TT2} have developed a theory counting stable and semistable Higgs pairs $(E,\phi)$ where $E$ is a torsion free rank $r$ coherent sheaf on $S$, and $\phi: E\to E\otimes K_S$ is a $\oO_S$-linear morphism called the Higgs field.  Let $N^{\perp}_L:=N_L^{\perp}(S,c)$ represent the moduli space of Gieseker stable Higgs pairs with Chern class $c=(r, c_1, c_2)\in H^*(S,\zz)$, fixed determinant $\det(E)=L$ and trace free $\phi$.  In \cite{TT}, the moduli space $N^{\perp}_L$ is proved to be isomorphic to the moduli space of stable two dimensional torsion sheaves on $X:=\Tot(K_S)$ supported on $S$.  Since $X$ is a Calabi-Yau threefold (not compact), the moduli space $N^{\perp}_L$ admits a symmetric obstruction theory in \cite{Behrend}.  Therefore there exists a virtual fundamental class $[N^{\perp}_L]^{\virt}\in H_0(N^{\perp}_L)$.  The space $N^{\perp}_L$ is not compact, but admits a $\cc^*$-action scaling the Higgs field and the fixed locus is compact.  The Vafa-Witten invariants are defined by virtual localization 
$$\VW_c(S)=\int_{[(N^{\perp}_L)^{\cc^*}]^{\virt}}\frac{1}{e(N^{\virt})}$$
where $e(N^{\virt})$ is the Euler class of the virtual normal bundle.  These invariants are called the $\SU(r)$-Vafa-Witten invariants. 
On the other hand the following invariants 
$$\vw_c(S)=\chi(N^{\perp}_L, \nu_{N})$$
are defined as Behrend's weighted Euler characteristic of $N^{\perp}_L$ weighted by the Behrend function 
$\nu_N: N^{\perp}_L\to \zz$.  These two invariants $\VW(S)$ and $\vw(S)$ are not the same for general surfaces, especially for general type surfaces, but for surfaces with $K_S<0$ and K3 surfaces, $\VW_c(S)=\vw_c(S)$, see \cite{TT}, \cite{MThomas}. 

From a general theory in \cite{JU}, the moduli space $N^{\perp}_L$ admits an orientation $K_{N^{\perp}_L}^{\frac{1}{2}}$. 
Let $\mMF_{N^{\perp}_L}^{\phi}$ be the global motive on $N^{\perp}_L$ obtained by gluing the local motive of vanishing cycles of the local function.  The motivic Vafa-Witten invariants are defined as:
\begin{defn}\label{defn_motivic_VW}
$$\vw^M_{c}(S)=\int_{N^{\perp}_L}[\mMF_{N^{\perp}_L}^{\phi}].$$
\end{defn}

In \cite{Thomas}, R. Thomas defined the K-theoretical Vafa-Witten invariants using the virtual structure sheaf  and an orientation of the moduli spaces.  His invariants are the K-theoretical refinement of the invariants $\VW(S)$. 
Our  motivic Vafa-Witten invariants are only the motivic refinement of the invariants $\vw(S)$. They give the same refinement for K3 surfaces.

The locus $N\subset N^{\perp}_L$ consisting of Higgs pairs $(E,\phi)$ such that the corresponding $E$ is stable is a quasi-projective subscheme of $N^{\perp}_L$.  The $\cc^*$-fixed locus of $N$ is just $M:=M_L(S,c)$, which is the moduli space of stable torsion free coherent sheaves with topological invariants $c$. This fixed component is the Instanton Branch of the  $\cc^*$-fixed locus of $N_L^{\perp}$, see \cite{TT}.  
The space $N\to M$ is an affine cone over $M$, and is just the situation in Theorem \ref{intro_thm_general_motivic_localization}. Thus Theorem \ref{intro_thm_general_motivic_localization} gives a calculation of the motivic Vafa-Witten invariants. 

Let $\mMF_N^{\phi}:=\mMF_{N^{\perp}_L}^{\phi}|_{N}$ be the restriction of the global motive on $N$. In this paper we calculate the motivic Vafa-Witten invariants for K3 surfaces.  We fix in $c_0=(r,c_1,c_2)$ the rank $r$ and $c_1$ are coprime, so that the moduli space 
$M_L(S,c_0)$ is an irreducible smooth variety.   We form the generating series 
$$\vw^M_{r,c_1}(S)=\sum_{c_2}\vw^M_{(r,c_1, c_2)}(S)q^{c_2}.$$
We have the following result:

\begin{thm}\label{thm_motivic_series_vw_K3}
Let $S$ be a smooth projective K3 surface.  Fix the rank $r$ and the first Chern class $c_1$ such that they are coprime. 
The motivic generating series of motivic Vafa-Witten invariants is given by:
$$\vw_{r, c_1}^M(S)=q^{r-\frac{1}{r}-\frac{1-r}{2r}c_1^2}\cdot\frac{1}{r}\sum_{j=0}^{r-1}e^{\pi i \frac{r-1}{r}j c_1^2}\cdot\prod_{n\geq 1}\left(1-\left(e^{\frac{2\pi i j}{r}}q^{\frac{1}{r}}\right)^n\right)^{-\ll^{-1}[S]}.$$
\end{thm}

We prove Theorem \ref{thm_motivic_series_vw_K3} using Theorem  \ref{intro_thm_general_motivic_localization}, and reduce the calculation to the general motive of $M$.  From Yoshioka \cite{Yoshioka}, $M$ is birational equivalent to the Hilbert scheme of points on $S$. Everything is reduced to the calculation of the motive of the Hilbert scheme of points on $S$, and then we use G\"ottsche's result in \cite{Gottsche}. 

Recall that there is a ring homomorphism $e: K_0(\Var_{\cc})\to \zz[u,v]$ from the Grothendieck ring of varieties to $\zz[u,v]$ given by the Hodge-Deligne polynomial of varieties.  The $\chi_y$-genus is just $e(y,1)$, i.e., $u=y, v=1$.  If we define 
$$\vw^{\chi}_{r,c_1}(S)=\sum_{c_2}\vw^{\chi}_{(r,c_1, c_2)}(S)q^{c_2}$$
as the generating series of the $\chi_y$-genus of the Vafa-Witten invariants, where 
$$\vw^{\chi}_{(r,c_1, c_2)}(S):=\chi_y(M_L(S, (r,c_1, c_2))).$$
As a corollary of Theorem \ref{thm_motivic_series_vw_K3} we calculate the $\chi_y$-genus:

\begin{cor}\label{cor_chi_y_series_vw_K3}
Let $S$ be a smooth projective K3 surface.  Fix the rank $r$ and the first Chern class $c_1$ such that they are coprime. 
Then the  $\chi_y$-genus of the  Vafa-Witten invariants is given by:
$$\vw_{r, c_1}^{\chi}(S)=q^{r-\frac{1}{r}-\frac{1-r}{2r}c_1^2}\cdot \frac{1}{r}\sum_{j=0}^{r-1}e^{\pi i \frac{r-1}{r}j c_1^2}\cdot\widetilde{\Delta}(e^{\frac{2\pi i j}{r}}q^{\frac{1}{r}}, y)^{-1},$$
where 
$$\widetilde{\Delta}(q, y):=\prod_{n\geq 1}\left(1-q^n\right)^{20}\left(1-y q^n\right)^{2}\left(1-y^{-1}q^n\right)^{2}.$$
\end{cor}
Thus we get similar result as in \cite{Thomas} where R. Thomas uses K-theoretical Vafa-Witten invariants to calculate the case of K3 surfaces, and also prove Conjecture 4.6 in \cite{GK} using motivic method.

\medskip
\noindent\textbf{Discussion.}  Our motivic series of the Vafa-Witten invariants give the same result as in \cite{Thomas} for K3 surfaces, since for a K3 surface $S$, $\VW(S)=\vw(S)$. As pointed out in \cite{TT}, the big Vafa-Witten invariants  $\VW(S)$ are the correct invariants for the S-duality. It is especially interesting for general type surfaces. 

Vafa-Witten actually predicted an S-transformation formula from the gauge group $\SU(r)$ to its Langlands dual $\SU(r)/\zz_r$.  In the case of K3 surfaces, the prime rank case is proved in \cite{Jiang_2019}, and higher rank case is studied in \cite{Jiang_Tseng}. The method we use is moduli space of optimal gerbe twisted sheaves on K3 surfaces.   The K-theoretical and elliptic genus version of the S-duality conjecture for K3 surfaces are proved in \cite{JM}.  The K-theoretical refinement 
version of the S-duality implies an S-duality transformation  for motivic Vafa-Witten invariants.

\medskip
\noindent\textbf{Outline.} This  note is organized as follows. In Section \ref{constr} we  review the construction of the cone  $N$, and prove that $N$ is a $d$-critical scheme.  We prove a motivic localization formula of the oriented $d$-critical scheme $N$ under $\Gm$-action in Section \ref{sec_motivic_localization_formula}, and apply it to get the weighted Euler characteristic of  the cone $N$. 
Finally in Section \ref{sec_motivic_VW} we apply the motivic localization formula to calculate the motivic generating series of the motivic Vafa-Witten invariants for K3 surfaces. 
\medskip

\noindent\textbf{Acknowledgements.} This paper is motivated by the study by Tanaka-Thomas \cite{TT}, \cite{TT2} on the Vafa-Witten invariants for projective surfaces.   I would like to thank R. Thomas for sending me the above papers before posting on arXiv.  Many thanks to B. Szendroi  for the valuable discussion on motivic invariants of Hilbert schemes of points on $\cc^3$,  and  S. Payne for the discussion on motivic nearby cycles and motivic Minor fibers via semi-algebraic sets. 
Y. J.  is partially supported by  NSF Grant DMS-1600997.
\medskip

\noindent\textbf{Notation.} Throughout the paper  we work for with a projective scheme $M$ over
$\kappa$ with perfect obstruction theory $E^{\bullet}\to\LL_M$. The one dimensional torus is denoted by  $\Gm$.
When we work on the motivic Vafa-Witten theory everything is over $\cc$ and $\Gm=\cc^*$.

\section{Preliminaries on the cone $N$} \label{constr}
We briefly recall the construction of the abelian cone $N$ in \cite{JT}. 

\subsection{Abelian cones}
Let  $F$ be coherent sheaf over $M$. There is an associated cone
$$
C(F):=\Spec\Sym^\bullet F\Rt{\pi_F}M
$$
over $M$. Cones of this form are called \emph{abelian} in \cite[Section 1]{BF}. The grading on $\Sym^\bullet F$ endows $C(F)$ with a $\Gm$-action 
$$\Gm\times C(F)\To C(F)$$
induced by the map
$$\Sym^\bullet F[x,x^{-1}]\longleftarrow\Sym^\bullet F\ \ $$
that takes $s\in \Sym^i F$ to $sx^i$. Its fixed locus is the zero section $M\subset C(F)$ defined by the ideal $\Sym^{\ge1}F$.

When $F$ is locally free $C(F)=\mathrm{Tot}(F^*)$ is the total space of the dual vector bundle. More generally, for any $F$, the fibre of $C(F)$ over a closed point $p\in M$ is the vector space $(F|_p)^*$. In fact $C(F)$ represents the functor from $M$-schemes to sets that takes $f\colon S\to M$ to $\Hom_S(f^*F,\O_S)$.

\subsection{The cone $N$}
We fix a perfect obstruction theory
$$
E\udot\To\LL_M
$$
of virtual dimension
$$
\vd:=\rk(E\udot)
$$
on the complex projective scheme $M$.

Applying the results of the last section to the obstruction sheaf $$\Ob_M\!:=h^1\big((E\udot)^\vee\big),$$ we define $\pi=\pi_N\colon N\to M$ to be the associated abelian cone,\footnote{Another way to describe $N$ is as the coarse moduli space of the vector bundle stack $h^1/h^0\big((E\udot)^\vee\big)$ of \cite[Section 2]{BF}.}
\beq{defnN}
N:=C(\Ob_M)=\Spec\Sym^\bullet(\Ob_M)\Rt\pi M.
\eeq

\subsection{Local model} 

Locally we may choose a presentation of $(M,E\udot)$ as the zero locus of a section $s$ of a vector bundle $E\to A$ over a smooth ambient space $A$, such that the resulting complex
$$
\big\{T_A|_M\Rt{ds}E|_M\big\} \quad\text{is}\quad
\big\{E_0\To E_1\big\}=(E\udot)^\vee.
$$
Then we have  $E\udot$ as $E^{-1}\to E^0$, we get the exact sequence
$$
E_0\stackrel{\phi}{\longrightarrow} E_1\To\Ob_M\To0.
$$
The resolution gives an exact sequence
$$
\phi(E_0)\otimes\Sym^{\bullet\;-1}E_1\To\Sym^\bullet E_1\To\Sym^\bullet \Ob_M\To0.
$$
That is,
$$
\Spec\Sym^\bullet \Ob_M\ \subset\ \Spec\Sym^\bullet E_1
$$
with ideal generated by $\phi(E_0)$. Letting $\tau$ denote the tautological section of $\pi_{E_1}^*E^{-1}$, this says that\smallskip
\beq{sentence}
C(\Ob_M) \text{ is cut out of }C(E_1)=\mathrm{Tot}\;(E^{-1}) \text{ by the section $\pi_{E_1}^*\phi^*(\tau)$ of } \pi_{E_1}^*E^0.
\eeq
Therefore $N=C(\Ob_M)$ is cut out of Tot$\;(E^*)|_M$ by the section $\pi_E^*(ds)^*(\tau)$ of $\pi_E^*\Omega_A|_M$. In turn Tot$\;(E^*)|_M$ is cut out of Tot$\;(E^*)$ by $\pi_E^*s$. Therefore the ideal of $N$ in the smooth ambient space Tot$\;(E^*)$ is
\beq{ideal}
\big(\pi_E^*s,\,\pi_E^*(Ds)^*(\tau)\big),
\eeq
where we have chosen any holomorphic connection $D$ on $E\to A$ by shrinking $A$ if necessary.

Thinking of the section $s$ of $E\to A$ as a linear function $\widetilde s$ on the fibres of Tot$\;(E^*)$, we find that its critical locus is $N$.

\begin{prop}\label{locCS}(\cite[Proposition 2.8]{JT})
$N\subset\mathrm{Tot}\;(E^*)$ is the critical locus of the function
$$
\widetilde s\colon\mathrm{Tot}\;(E^*)\to\kappa\;.
$$
\end{prop}

\section{The global sheaf of vanishing cycles  on $N$}

\subsection{$d$-critical schemes}

Let us first recall the notion of $d$-critical schemes introduced in \cite{Joyce}.   For any scheme $N$, Joyce \cite[Theorem 2.1]{Joyce} proves that there exists a canonical sheaf of $\cc$-vector spaces $\sS_N$ on $N$.  This sheaf $\sS_N$ satisfies the following property: 
for any Zariski open subset $\widetilde{R}\subset N$ and a closed embedding $i: \widetilde{R}\hookrightarrow \widetilde{A}$ into a smooth scheme 
$\widetilde{A}$, there is an exact sequence:
\begin{equation}\label{eqn_exact_d_critical_scheme}
0\to \sS_N|_{\widetilde{R}}\longrightarrow \oO_{\widetilde{A}}/I^2\stackrel{d_{DR}}{\longrightarrow} \Omega_{\widetilde{A}}/I\cdot \Omega_{\widetilde{A}}\to 0,
\end{equation}
where $I\subset \oO_{\widetilde{A}}$ is the ideal sheaf of $\widetilde{R}$ and $d_{DR}$ is the de-Rham differential.  Joyce shows that there is a natural decomposition 
$$\sS_N=\sS_N^{0}\oplus \cc_{N}$$
where $\cc_N$ is the constant sheaf on $N$.  This sheaf $\sS^{0}_N$, when restricted to $\widetilde{R}$, is the kernel of the composition
$$\sS_N|_{\widetilde{R}}\hookrightarrow \oO_{\widetilde{A}}/I^2\twoheadrightarrow \oO_{\widetilde{R}^{\red}}$$
and $\widetilde{R}^{\red}$ is the reduced scheme of $\widetilde{R}$.
We recall the $d$-critical scheme structure of Joyce.

\begin{defn}\label{defn_d_critical_scheme}(\cite[Definition 2.5]{Joyce})
A $d$-critical scheme is given by a pair $(N, s)$, where $N$ is a scheme and $s\in \Gamma(N, \sS_N^{0})$ is a section such that for any point 
$x\in N$, there is an open neighborhood $x\in \widetilde R\subset N$,  a closed embedding $i:  \widetilde{R}\hookrightarrow \widetilde{A}$ into a smooth scheme $\widetilde{A}$, and a regular function $f:  \widetilde{A}\to \mathbb{A}^1$ such that $s|_{\widetilde{R}}=f+I^2$.  We call the data
$(\widetilde{R}, \widetilde{A}, f, i)$ a $d$-critical chart. 
\end{defn}

Therefore a $d$-critical scheme $(N,s)$ is roughly understood as locally the critical locus of some regular function $f$ on a smooth scheme, and the section $s$
remembers the local regular function $f$.  In particular if $N=\Crit(f)$ is the critical locus of a regular function $f: \widetilde{A}\to \mathbb{A}^1$, then $N$ is a $d$-critical scheme and the sheaf $\sS_N$ is given by (\ref{eqn_exact_d_critical_scheme}).

\subsection{Semi-symmetric obstruction theory}\label{subsec_semi-symmetric_obstruction_theory}

In this section we review a bit of the semi-symmetric obstruction theory in \cite{CL}, \cite{Jiang5}.

\begin{defn}(\cite[Definition 3.1]{CL})\label{defn_semiPOT}
A semi-perfect obstruction theory of $N$ consists of a covering $\{N_{\alpha}\}_{\alpha\in\Lambda}$ of 
$N$ by affine schemes, and truncated perfect obstructin theories
$$\phi_{\alpha}: E_{\alpha}\to L_{N_\alpha}, \alpha\in\Lambda$$
such that 
\begin{enumerate}
\item for any pair $\alpha, \beta\in \Lambda$ there exists an isomorphism 
\begin{equation}\label{isomorphism_psi_alphabeta}
\psi_{\alpha\beta}: h^1(E_{\alpha}^\vee)|_{N_{\alpha\beta}}\stackrel{\cong}{\longrightarrow}
h^1(E_{\beta}^\vee)|_{N_{\alpha\beta}}
\end{equation}
such that the collections $(h^1(E_{\alpha}^\vee), \psi_{\alpha\beta})$ forms a descent data of sheaves.
\item for any pair $\alpha, \beta\in \Lambda$, the obstruction theories 
$\phi_{\alpha}|_{N_{\alpha\beta}}$ and $\phi_{\beta}|_{N_{\alpha\beta}}$ are $\nu$-equivalent via $\psi_{\alpha\beta}$, which means that 
$\psi_{\alpha\beta}$ is compatible with the infinitesimal lifting problem for the obstruction theories, see \cite[Definition 2.9]{CL}, and \cite[Definition 3.5]{Jiang5}.

\end{enumerate}
A symmetric semi-perfect obstruction theory on $N$ is a semi-perfect obstruction theory $\phi=\{\phi_\alpha, N_\alpha, E_\alpha, \psi_{\alpha\beta}\}_{\alpha\in\Lambda}$ for $N$ such that each $\phi_{\alpha}$ is symmetric in the sense of \cite{Behrend}. 

A (symmetric) perfect obstruction theory is a (symmetric) semi-perfect obstruction theory.  
\end{defn}

For a symmetric semi-perfect obstruction theory $\phi=\{\phi_\alpha, N_\alpha, E_\alpha, \psi_{\alpha\beta}\}_{\alpha\in\Lambda}$ for $N$, the local virtual cycles glue to give a virtual fundamental cycle $[N]^{\virt}\in A_0(N)$.  Behrend's theorem equating the weighted Euler characteristic $\chi(N,\nu_N)$ with the virtual count 
$\int_{[N]^{\virt}}1$ is still true, see \cite[Theorem 3.8]{Jiang5} if $N$ is proper.  If $N$ is non-proper, it does not make sense for the integral, but 
 $\chi(N,\nu_N)$ exists as an invariant. 
 
\begin{prop}\label{prop_symmetric_semi_POT}(\cite[Theorem 4.6]{Jiang5})
There exists a symmetric semi-oerfect obstruction theory $\phi=\{\phi_\alpha, N_\alpha, E_\alpha, \psi_{\alpha\beta}\}_{\alpha\in\Lambda}$ for a 
$d$-critical scheme $(N,s)$. 
\end{prop}

\subsection{$N$ is a $d$-critical scheme}\label{subsec_N_d_critical_scheme}
We show that our cone $N$ in (\ref{defnN}) is a $d$-critical scheme. 

\begin{prop}\label{prop_d_critical_scheme_N}
The scheme $N$  in (\ref{defnN}) is a  $d$-critical scheme in the sense of \cite[Definition 2.5]{Joyce}. 
\end{prop}
\begin{proof}
We show that there exists a section $s\in \Gamma(N, \sS_N^{0})$ satisfying the properties in  Definition \ref{defn_d_critical_scheme}. 

Let $x\in N$ be a point.  Recall the projection of the cone $\pi: N\to M$. 
We denote by $\overline{x}=\pi(x)$ and $\overline{x}\in M$. 
let $\overline{x}\in R$ be an open subset in $M$. Suppose that  $R\subset A$ is a local embedding into a smooth scheme $A$. 
Let $\widetilde{R}=\pi^{-1}(R)$ and $\widetilde{R}\to R$ the projection.  We have $x\in \widetilde{R}$. 
We may shrink $R$ and $A$ if necessary so that  the vector bundle $E^*$ is trivial and there is a relative embedding 
$$\widetilde{R}\hookrightarrow \widetilde{A}:=A\times \aaa_\kappa^{\rank{E^*}}.$$
Then from the proof of \cite[Proposition 2.8]{JT}, 
one can work in local coordinates $x_i$ for $A$.  Take a basis of sections $e_j$ of $E$, we get a dual basis $f_j$ for $E^*$ and coordinates $y_j$ on the fibers of Tot$\;(E^*)$.
Then we can write $f=\sum_js_je_j,\ \tau=\sum_jy_jf_j$ and
$$\widetilde f=\sum_j s_jy_j.$$
Therefore
$$
d\;\widetilde f=\sum_jy_jds_j+\sum_js_jdy_j
=\big\langle\tau,\pi_E^*Df\big\rangle+\sum_js_jdy_j
$$
with zero scheme defined by the ideal
$$
\big(\pi_{E^*}^*(Df)^*(\tau),\,\pi_E^*s_1,\,\pi_E^*s_2,\,\ldots\big),
$$
which  is the same as \eqref{ideal}.

Then $\Crit(\s)\cong \widetilde{R}$. 
Let $\widetilde{i}: \widetilde{R}\hookrightarrow  \widetilde{A}$ be the inclusion. Then we get a $d$-critical chart 
$$(\widetilde{R},  \widetilde{A}, \widetilde f, \widetilde{i}).$$
Let $\widetilde{I}_{\tR,\tA}\subset \oO(\widetilde{A})$ be the ideal of $\widetilde{R}$ in $\widetilde{A}$.
Then we have an exact sequence:
$$0\rightarrow \mathcal{S}_{N}|_{\widetilde{R}}\stackrel{\iota_{\widetilde{R},\widetilde{A}}}{\longrightarrow}
\frac{\widetilde{i}^{-1}(\oO_{\widetilde{A}})}{\widetilde{I}_{\tR,\tA}^2}\stackrel{d_{DR}}{\longrightarrow}
\frac{\widetilde{i}^{-1}(\Omega_{\widetilde{A}})}{\widetilde{I}_{\tR,\tA}\cdot \widetilde{i}^{-1}(\Omega_{\widetilde{A}})}.$$
Then  locally the section $s\in \Gamma(N,\sS_{N}^{0})$ is given by 
$$s|_{\widetilde{R}}=\widetilde f+ \widetilde{I}_{\tR,\tA}^2.$$

The  $d$-critical chart exists around every point $x\in \widetilde{R}\subset N$.  
To show $s|_{\widetilde{R}}$ glue to give a section $s\in \Gamma(N,\sS_N^{0})$, 
Let $(\widetilde{S},  \widetilde{B}, \widetilde g, \widetilde{j})$ be another chart such that 
$x\in \widetilde S$, where 
$$\tS\hookrightarrow \tB:=B\times\aaa_{\kappa}^{\rank E^\prime}$$
for a smooth scheme $B$ and vector bundle $E^\prime\to B$ over $B$. 
Let $\phi: \tR\to \tS$ be an embedding of open subsets of $N$ such that in the diagram
\[
\xymatrix{
\tR\ar[r]^{}\ar[d]_{\phi}& \widetilde{A}\ar[d]^{\Phi}\\
\tS\ar[r]^{}& \widetilde{B}
}
\]
$\Phi$ is a smooth embedding since shrinking $A$ and $B$ if necessary, we have 
$ \widetilde{B} =\widetilde{A}\times \aaa_{\kappa}^{m}$ for some $m\in \zz_{>0}$.  
This can be seen as follows.  Let $\pi: \widetilde{S}\to S\subset M$ and $\pi: \widetilde{R}\to R\subset M$ be the projections, one can take local sections
$(s_1^\prime, \cdots, s_{r^\prime}^{\prime})$
of the bundle $E^\prime \to B$ with zero locus $S$, such that  its restriction to 
$E\to A$ gives the zero locus $R$ of the local sections $(s_1,\cdots, s_r)$ of $E\to A$.  Thus shrinking $A$ and $B$ if necessary, we may assume 
$ \widetilde{B} =\widetilde{A}\times \aaa_{\kappa}^{m}$ for some $m\in \zz_{>0}$. And the local functions $g=\sum_i s_i^\prime\cdot y_i^\prime$ restrict to give the local function $f$. 
Then we have the following diagram:
\begin{equation}\label{diagram_1}
\xymatrixcolsep{3pc}\xymatrix{
0\ar[r]&\left(\mathcal{S}_{N}|_{\widetilde{S}}\right)|_{\tR}\ar[r]^{\iota_{\tS,\widetilde{B}}}\ar[d]^{\id}& \frac{\widetilde{j}^{-1}(\oO_{\widetilde{B}})}{I_{\tS,\widetilde{B}}^2}|_{\tR}\ar[r]^{d}\ar[d]^{i^{-1}(\Phi^{\#})}& \frac{\widetilde{j}^{-1}(\Omega_{\widetilde{B}})}{I_{\tS,\widetilde{B}}\cdot \widetilde{j}^{-1}(\Omega_{\widetilde{B}})}|_{\tR}\ar[r]\ar[d]^{i^{-1}(d\Phi)}&0 
\\
0\ar[r]&\mathcal{S}_{N}|_{\tR}\ar[r]^{\iota_{\tR,\widetilde{A}}}& \frac{\widetilde{i}^{-1}(\oO_{\widetilde{A}})}{I_{\tR,\widetilde{A}}^2} \ar[r]^{d}& 
\frac{\widetilde{i}^{-1}(\Omega_{\widetilde{A}})}{I_{\tR,\widetilde{A}}\cdot \widetilde{i}^{-1}(\Omega_{\widetilde{A}})}\ar[r]&0.
}
\end{equation}
The local function $\widetilde g$ restrict to give the local function $\widetilde f$,  therefore the element $g+I^2_{\tS,\widetilde{B}}$ restricts to give the element $f+I^2_{\tR,\widetilde{A}}$.

Now for any two $d$-critical charts $(\widetilde{R},  \widetilde{A}, \widetilde f, \widetilde{i})$ and $(\widetilde{S},  \widetilde{B}, \widetilde g, \widetilde{j})$ of $N$, 
then there is a chart  $(\widetilde R\cap \widetilde S, \widetilde W, \widetilde h,, \widetilde k)$ on the intersection $\widetilde R\cap \widetilde S$, such that there are embeddings of the charts 
$$(\widetilde R\cap \widetilde S, \widetilde W, \widetilde h,, \widetilde k)\hookrightarrow (\widetilde{R},  \widetilde{A}, \widetilde f, \widetilde{i})$$ 
and 
$$(\widetilde R\cap \widetilde S, \widetilde W, \widetilde h,, \widetilde k)\hookrightarrow (\widetilde{S},  \widetilde{B}, \widetilde g, \widetilde{j})$$ 
such that the local functions $\widetilde f$ and  $\widetilde g$ gives the local function $\widetilde h$. 
Hence there is a section 
$s\in H^0(\mathcal{S}^{0}_{N})$ and $(N,s)$ is a $d$-critical scheme. 
\end{proof}

\subsection{Orientation on $d$-critical schemes}

Let $(N,s)$ be a $d$-critical scheme.  There exists a line bundle $K_{N,s}$ on $N^{\red}$, called the virtual canonical line bundle.  
From 
 \cite[Theorem 2.28 ]{Joyce}  $K_{N,s}$ is the unique line bundle on $N$ such that 
on the $d$-critical chart $(\widetilde{R},  \widetilde{A}, \s, \widetilde{i})$, there is a natural isomorphism 
$$\iota: K_{N,s}|_{\widetilde{R}^{\red}}\stackrel{\cong}{\longrightarrow} 
\widetilde{i}^{\star}(K_{\widetilde{A}}^{\otimes 2})|_{\widetilde{R}^{\red}}.$$

\begin{defn}\label{defn_orientation}(\cite[Definition 2.31]{Joyce})
A $d$-critical  scheme $(N,s)$ is called $oriented$ if the line bundle  $K_{N,s}$ admits a square root $K_{N,s}^{\frac{1}{2}}$
on $N^{\red}$ such that 
$$\left(K_{N,s}^{\frac{1}{2}}\right)^{\otimes 2}\cong K_{N,s}.$$ 
We call such  a $d$-critical  scheme $(N,s)$ an oriented $d$-critical scheme. 
\end{defn}

\begin{defn}(\cite[Definition 2.8]{MT})
A $d$-critical scheme $(N,s)$ is Calabi-Yau if the virtual canonical line bundle $K_{N,s}\cong \oO_{N^{\red}}$ is trivial.  An oriented 
$d$-critical scheme $(N,s, K_{N,s}^{\frac{1}{2}})$ is Calabi-Yau of $K_{N,s}^{\frac{1}{2}}\cong  \oO_{N^{\red}}$.
\end{defn}

\begin{rmk}
It is interesting to see if our cone $N$ is a Calabi-Yau $d$-critical scheme.  Maulik-Toda \cite{MT} introduced the Calabi-Yau structure on a projective morphism 
$N\to B$ (for a scheme $B$) at a point $b\in B$, which means there exists an open neighborhood $b\in U\subset B$ such that $(N|_{U}, s|_{U})$ is Calabi-Yau.  The projection morphism $\pi: N\to M$ is not projective, therefore it is not interesting to have Calabi-Yau structure at a point $x\in M$. 
\end{rmk}

The cone $N$ in  (\ref{defnN}) admits a $\Gm$-action whose fixed point loci is $M\subset N$.  As pointed out in \cite{JT}, it is not known if  the scheme $N$ admits a symmetric obstruction theory of \cite{Behrend}, but there is one locally since $N$ is the critical locus of a regular function on a higher dimensional smooth scheme.  In \cite{JT}, the authors assume that the scheme $M$ is the underlying scheme of a quasi-smooth derived scheme, and $N$ is the $(-1)$-shifted cotangent bundle of $M$, therefore admits a symmetric obstruction theory.  

From Proposition \ref{prop_symmetric_semi_POT}, the $d$-critical scheme $(N,s)$ admits a symmetric semi-perfect obstruction theory
$\phi=\{\phi_\alpha, N_\alpha, E_\alpha, \psi_{\alpha\beta}\}_{\alpha\in\Lambda}$ as in Definition \ref{defn_semiPOT}. 
The $\Gm$-action on $N$ induces an action on $\phi=\{\phi_\alpha, N_\alpha, E_\alpha, \psi_{\alpha\beta}\}_{\alpha\in\Lambda}$ and makes it a $\Gm$-equivariant symmetric semi-perfect obstruction theory. 
Note that if $N=\bigcup_{\alpha\in\Lambda}N_{\alpha}$, then $M=\bigcup_{\alpha\in\Lambda}(N_{\alpha}\cap M)$.
Here is a generalization of 
\cite[Proposition 2.6]{Thomas}:

\begin{prop}\label{prop_Thomas_orientation} 
The canonical line bundle  $K_{N,s}$ 
when restricted to the fixed point locus $M$ has  a canonical square root such that for each $\alpha\in \Lambda$, 
$$K_{N,s}^{\frac{1}{2}}|_{\pi(N_{\alpha})}=\det(E_{\alpha}|_{M\cap N_{\alpha}})^{\geq 0}\mathfrak{t}^{\frac{1}{2}\rk_{\geq 0}},$$
where $(E_{\alpha}|_{M\cap N_{\alpha}})^{\geq 0}$ denotes the part of  $E_{\alpha}|_{M\cap N_{\alpha}}$ with nonnegative $\Gm$-weights, and $\rk_{\geq 0}$ is its rank. 
\end{prop}
\begin{proof}
We can take the chart $N_{\alpha}$ to be a $d$-critical chart
$(N_{\alpha}, \widetilde f, \widetilde A, \widetilde i)$, then  the symmetric obstruction theory $E_{\alpha}$ is given by
$$E_{\alpha}=\Big[T_{\widetilde A}\stackrel{d\circ d\widetilde f^{\vee}}{\longrightarrow} \Omega_{\widetilde A}\Big]$$
Therefore $K_{N,s}|_{N^{\red}_{\alpha}}=\det (E_{\alpha})=\widetilde i^{\star}K_{\widetilde A}^{\otimes 2}|_{N_{\alpha}^{\red}}$.  The local calculation is the same as in the proof of \cite[Proposition 2.6]{Thomas}.  Since all the local data glue, we are done. 
\end{proof}

\subsection{Motivic vanishing cycles}\label{subsec_motivic_vanishing_cycle}

In this section we review the motivic vanishing cycles for $d$-critical schemes constructed  in \cite{BJM}. 
For a scheme $N$, let $K_0(\Var_N)$ be the Grothendieck ring of schemes over the scheme $N$.  It is the abelian group generated by symbols 
$[T]$ for $T\to N$ an $N$-scheme with relations $[S]=[T]$ if $S\cong T$ as $N$-schemes, and if $T\subseteq S$ is a closed $X$-subscheme, then
$[S]=[T]+[S\setminus T]$.  The ring structure on  $K_0(\Var_N)$ is given by $[S]\cdot [T]=[S\times_{N}T]$.
Let $\aaa_N^1$ be the $N$-scheme $\pi_N: \aaa^1\times N\to N$ and we denote it by $\ll:=[\aaa_N^1]$.

For all the positive integers $n\in\zz_{>0}$, we consider the cyclic group $\mu_n$.  These groups $\mu_n$ form a projective system with respect to the maps 
$\mu_{nd}\to \mu_n$ mapping $x\mapsto x^d$ for all $d, n\in\zz_{>0}$.  We let $\hat{\mu}$ be the projective limit of the groups $\mu_n$.  $\hat{\mu}$ is a pro-scheme.  Let $T\to N$ be an $N$-scheme.  A good $\mu_n$-action on $T$ is a group action 
$\sigma: \mu_n\times T\to T$ such that each orbit is contained in an affine subscheme of $T$. If $T$ is quasi-projective then any action of $\mu_n$ on $T$ is a good action.   A good $\hat{\mu}$-action on $T$ is an action $\hat{\sigma}: \hat{\mu}\times T\to T$ which factors through a good $\mu_n$-action for some $n$. 
The trivial action is denoted by $\hat{\iota}: \mu_n\times T\to T$ which is good. 

The equivariant  Grothendieck ring of schemes $K^{\hat{\mu}}_0(\Var_N)$ is, as an abelian group, generated by the symbols 
$[T,\hat{\sigma}]$ where $T\to N$ is an $N$-scheme with a good $\hat{\mu}$-action $\hat{\sigma}$. The relations are given by:

$[T,\hat{\sigma}]=[S,\hat{\tau}]$ is $T, S$ are isomorphic as $N$-schemes with $\hat{\mu}$-actions;

$[T,\hat{\sigma}]=[S,\hat{\sigma}|_{S}]+ [T\setminus S, \hat{\sigma}|_{T\setminus S}]$ if $S\subseteq T$ is a closed, 
$\hat{\mu}$-invariant $N$-subscheme of $T$; and 

$[T\times \aaa^n, \hat{\sigma}\times\hat{\tau}_1]=[T\times \aaa^n, \hat{\sigma}\times\hat{\tau}_2]$ for any linear $\hat{\mu}$-actions $\hat{\tau}_1, \hat{\tau}_2$
on $\aaa^n$.

The ring structure on  $K^{\hat{\mu}}_0(\Var_N)$  is given by 
$[T,\hat{\sigma}]\cdot [S,\hat{\tau}]=[T\times_{N}S, \hat{\sigma}\times\hat{\tau}]$.   Still let 
$\ll=[\aaa_N^1, \hat{\iota}]$ be the Lefschetz motive in $K^{\hat{\mu}}_0(\Var_N)$. 
We define 
$$\mM_{N}^{\hat{\mu}}=K^{\hat{\mu}}_0(\Var_N)[\ll^{-1}]$$
to be the ring by inverting $\ll$.  If $N=\Spec \kappa$, then   $K^{\hat{\mu}}_0(\Var_N)$ and $\mM_{N}^{\hat{\mu}}$ are written as 
$K^{\hat{\mu}}_0(\Var_{\kappa})$ and $\mM_{\kappa}^{\hat{\mu}}$. 
Some properties of the  $K^{\hat{\mu}}_0(\Var_N)$ and $\mM_{N}^{\hat{\mu}}$ are given in \cite[\S 2]{BJM}, we refer the reader to \cite{BJM} for more details. 
\cite{BJM} also introduced a product $\odot$ on  $K^{\hat{\mu}}_0(\Var_N)$ and $\mM_{N}^{\hat{\mu}}$ which we recall here:

\begin{defn}\label{defn_odot_product}(\cite[Definition 2.3]{BJM})
Let $[T, \widehat{\sigma}], [S,\widehat{\tau}]$ be two elements in $K_0^{\hat{\mu}}(\Var_{N})$ or $\mM_{N}^{\hat{\mu}}$. 
Then there exists $n\geq 1$ such that the $\hat{\mu}$-actions $\widehat{\sigma}, \widehat{\tau}$ on $T, S$ factor through $\mu_n$-actions
$\sigma_n, \tau_n$.  Define $J_n$ to be the Fermat curve
$$J_n=\{(t,u)\in (\aaa^1\setminus \{0\})^2: t^n+u^n=1\}.$$
Let $\mu_n\times\mu_n$ act on $J_n\times(T\times_{N}S)$ by
$$(\alpha,\alpha^\prime)\cdot ((t,u),(v,w))=((\alpha\cdot t, \alpha^\prime\cdot u), (\sigma_n(\alpha)(v), \tau_n(\alpha^\prime)(w))).$$
Write $J_n(T,S)=(J_n\times (T\times_{N}S))/(\mu_n\times\mu_n)$ for the quotient $\kappa$-scheme, and 
define a $\mu_n$-action $v_n$ on $J_n(T,S)$ by
$$v_n(\alpha)((t,u), v,w)(\mu_n\times\mu_n)=((\alpha\cdot t, \alpha\cdot u),v,w)(\mu_n\times\mu_n).$$
Let $\hat{v}$ be the induced good $\hat{\mu}$-action on $J_n(T,S)$, and set
$$[T, \widehat{\sigma}]\odot [S,\widehat{\tau}]=(\ll-1)\cdot [(T\times_{N}S/\mu_n, \hat{\iota})]-[J_n(T,S),\hat{v}]$$
in  $K_0^{\hat{\mu}}(\Var_{N})$ or $\mM_{N}^{\hat{\mu}}$. This defines a commutative, associative product on  $K_0^{\hat{\mu}}(\Var_{N})$,  $\mM_{N}^{\hat{\mu}}$.
\end{defn}

In \cite{BJM}, for any $\kappa$-schemes $N, P$, there are products
$$\boxdot: K_0^{\hat{\mu}}(\Var_{N})\times K_0^{\hat{\mu}}(\Var_{P})\to K_0^{\hat{\mu}}(\Var_{N\times P}); \quad 
\boxdot: \mM_{N}^{\hat{\mu}}\times \mM_{P}^{\hat{\mu}}\to \mM_{N\times P}^{\hat{\mu}}$$
by mapping 
$$([T, \hat{\sigma}], [S,\hat{\tau}])\mapsto [T\times S, \hat{\sigma}\times \hat{\tau}].$$
If $P=\Spec(\kappa)$, then $N\times\Spec(\kappa)=N$, and $\boxdot$ makes $K_0^{\hat{\mu}}(\Var_{N})$,  $\mM_{N}^{\hat{\mu}}$ 
into modules over $K_0^{\hat{\mu}}(\Var_{\kappa})$, $\mM_{\kappa}^{\hat{\mu}}$.

As in \cite{BJM}, we define 
$\ll^{\frac{1}{2}}$ in $K_0^{\hat{\mu}}(\Var_{N})$ or $\mM_{N}^{\hat{\mu}}$ by:
$$\ll^{\frac{1}{2}}=[N,\hat{\iota}]-[N\times\mu_2,\hat{\rho}],$$
where $[N,\hat{\iota}]$ with trivial $\hat{\mu}$-action $\hat{\iota}$ is the identity in  $K_0^{\hat{\mu}}(\Var_{N})$ or $\mM_{N}^{\hat{\mu}}$,
and $N\times\mu_2$ is the two copies of $N$ with the nontrivial $\hat{\mu}$-action $\hat{\rho}$ induced by the left action of $\mu_2$ on itself, exchanging the two copies of $N$.  Then $\ll^{\frac{1}{2}}\odot\ll^{\frac{1}{2}}=\ll$. 

Now let $(N,s)$ be the $d$-critical scheme for the cone $N$ in Proposition \ref{prop_d_critical_scheme_N}.  Let $(\widetilde R, \widetilde A, \widetilde f, \widetilde i)$ be a $d$-critical chart of $(N,s)$.  Then the motivic vanishing cycle $\mMF_{\widetilde A, \widetilde f}^{\phi}\in \mM_{\Crit(\widetilde f)}^{\hat{\mu}}$   is defined as follows. 
The restriction $\widetilde f|_{\widetilde R}: \widetilde R\to \aaa^1$ is locally constant on $\widetilde R^{\red}$ and $\widetilde f(\widetilde R)$ is finite. 
We have $\widetilde R=\bigsqcup_{c\in \widetilde f(\widetilde R)} \widetilde R_c$ with $\widetilde R_c=\widetilde R\cap \widetilde A_c$ where $\widetilde A_c=\widetilde f^{-1}(c)\subset \widetilde A$. 
Then from \cite[Formula (2.8)]{BJM}, the motivic vanishing cycle $\mMF_{\widetilde A, \widetilde f}^{\phi}\in \mM_{\Crit(\widetilde f)}^{\hat{\mu}}$ is defined by 
\begin{equation}\label{eqn_motivic_vanishing_cycle_local}
\mMF_{\widetilde A, \widetilde f}^{\phi}|_{\widetilde R_c}=\ll^{-\frac{\dim \widetilde A}{2}}\odot 
([\widetilde A_c, \hat{\iota}]-\mMF_{\widetilde A, \widetilde f-c})|_{\widetilde R_c}\in \mMF^{\hat{\mu}}_{\widetilde R_c}
\end{equation}
for each $c\in \widetilde f(\widetilde R)$, 
where 
for the function $\widetilde f-c$ on $\widetilde A$,  $\mMF_{\widetilde A, \widetilde f-c}$ is the motivic nearby cycle 
$$\mMF_{\widetilde A, \widetilde f-c}:=\sum_{\emptyset \neq I\subseteq J}
(1-\ll)^{|I|-1}[\widetilde{E}_{I}^{\circ}\to \widetilde A_0, \hat{\rho}_{I}].$$
Here we take $p: \widetilde B\to \widetilde A$ be the resolution of singularities of $\widetilde f$.  Let 
$E_i, i\in J$ be the irreducible components of $p^{-1}(\widetilde A_0)$.  For each $i\in J$, let $N_i$ be the multiplicity of $E_i$ in the divisor of 
$\widetilde f\circ p$ on $\widetilde B$, and $\nu_i-1$ the multiplicity of $E_i$ in the divisor of $p^*(dx)$, where $dx$ is a local on vanishing 
volume form at any point of $p(E_i)$. 
For any $I\subset J$, let $E_I^{\circ}=(\cap_{i\in I}E_i)\setminus (\cup_{i\in J\setminus I}E_j)$. Let $m_I=\gcd(N_i)_{i\in I}$. There exist an unramified Galois cover 
$\widetilde{E}_{I}^{\circ}\to E_I^{\circ}$, with Galois group $\mu_{m_I}$ introduced in \cite[Definition 2.9]{BJM}.

\cite[\S 2.5]{BJM} also defined the motive of principle $\zz_2$-bundles. 
The reason is that the local regular functions $\widetilde f$, when apply the change of variables of the local coordinates, has extra terms by a quadratic polynomial, which determines a  principal $\zz_2$-bundle on $N$. 
Let $\zz_2(N)$ be the abelian group of isomorphism classes $[P]$ of principal $\zz_2$-bundles $P\to N$, where the multiplication is given by $[P]\cdot [Q]=[P\otimes_{\zz_2}Q]$  is identity is $[N\times \zz_2]$.  Each element in  $\zz_2(N)$ is self-inverse, and has order $1$ or $2$.  If $P\to N$ is a principal $\zz_2$-bundle  over $N$, then we have the motive of $P$ defined as:
$$\Upsilon(P)=\ll^{-1/2}\odot ([N,\hat{\iota}]-[P, \hat{\rho}])$$
where $\hat{\rho}$ is the $\hat{\mu}$-action on $P$ induced by the $\mu_2$-action on $P$ from the principal $\zz_2$-bundle structure.  If $P=N\times \zz_2$ is the trivial bundle, then 
$\Upsilon(N\times \zz_2)=[N,\hat{\iota}]$.
In \cite{BJM}, the authors defined the ring $\overline{\mM}^{\hat{\mu}}_{N}$, which is  roughly the quotient of  $\mM^{\hat{\mu}}_{N}$ by the relation 
$\Upsilon(P\otimes_{\zz_2}Q)-\Upsilon(P)\odot \Upsilon(Q)=0$.  In order to let the ring $\overline{\mM}^{\hat{\mu}}_{N}$ with product $\odot$ have pushforward property $\phi_*: \overline{\mM}^{\hat{\mu}}_{N}\to \overline{\mM}^{\hat{\mu}}_{P}$ under the morphism $\phi: N\to P$.  The authors define for a scheme $P$, 
the ideal $I^{\hat{\mu}}_{P}$ in $(\mM^{\hat{\mu}}_{P},\odot)$ generated by the elements 
$$\phi_*(\Upsilon(P\otimes_{\zz_2}Q)-\Upsilon(P)\odot \Upsilon(Q))$$
for all morphisms $N\to P$ and principal $\zz_2$-bundles $P, Q\to N$.   Then define $\overline{\mM}^{\hat{\mu}}_{P}$  to be the quotient 
$\mM^{\hat{\mu}}_{N}/I^{\hat{\mu}}_{P}$ with the multiplication $\odot$.  When $P=N$, and $\phi$ is identity, we get the ring  $\overline{\mM}^{\hat{\mu}}_{N}$. The motivic vanishing cycles $\mMF_{\widetilde R, \widetilde A}^{\phi}$, $\ll, \ll^{1/2}$ and $\Upsilon(P)$ are in the ring
$\overline{\mM}^{\hat{\mu}}_{N}$. 

Finally we recall the result in \cite{BJM} about the motivic vanishing cycles for an oriented $d$-critical scheme. 

\begin{thm}(\cite[Theorem 5.10]{BJM})
For our $d$-critical scheme $(N,s)$ in Proposition \ref{prop_d_critical_scheme_N},  assume that there exists an orientation 
$K_{N,s}^{1/2}$. Then there exists a unique motive $\mMF_{N,s}^{\phi}\in \overline{\mM}^{\hat{\mu}}_{N}$, called the motivic vanishing cycles of $N$, such that for each $d$-critical chart $(\widetilde R, \widetilde A, \widetilde f, \widetilde i)$, we have 
$$\mMF_{N,s}^{\phi}|_{\widetilde R}=\widetilde i^{*}(\mMF_{\widetilde A, \widetilde f}^{\phi})\otimes \Upsilon(Q_{\widetilde R, \widetilde A, \widetilde f, \widetilde i}) \quad  \text{in~} \overline{\mM}^{\hat{\mu}}_{\widetilde R}$$
where $Q_{\widetilde R, \widetilde A, \widetilde f, \widetilde i}\to \widetilde R$ is the principal $\zz_2$-bundle parametrizing local isomorphisms 
$\alpha: K_{N,s}^{1/2}|_{\widetilde R^{\red}}\to \widetilde i^* (K_{\widetilde A})|_{\widetilde R^{\red}}$ with 
$\alpha\otimes\alpha=\iota_{\widetilde R, \widetilde A, \widetilde f, \widetilde i}$ where 
$\iota_{\widetilde R, \widetilde A, \widetilde f, \widetilde i}: K_{N,s}|_{\widetilde R^{\red}}\to \widetilde i^* (K^{\otimes 2}_{\widetilde A})|_{\widetilde R^{\red}}$ is the local isomorphism. 
\end{thm}

\begin{rmk}
The orientation in Definition \ref{defn_orientation} for the moduli space of stable coherent sheaves on Calabi-Yau threefolds has been constructed in \cite{JU}.
\end{rmk}

\section{A  motivic localization formula}\label{sec_motivic_localization_formula}

We prove a version of the $\Gm$-localization formula for the global motive $\mMF_{N,s}^{\phi}$ for the oriented 
$d$-critical scheme $(N,s)$ in   Proposition \ref{prop_d_critical_scheme_N}  under $\Gm$-action scaling the fibers. 
In \cite[Theorem 7.17]{Jiang3}, the motivic localization formula for oriented  formal $d$-critical schemes and $d$-critical non-archimedean 
$\kk=\kappa(\!(t)\!)$-analytic spaces  was proved  using motivic integration for formal schemes. In \cite{Jiang3},  we assume that the $\Gm$-action on $X$ preserves the orientation $K_{N,s}^{\frac{1}{2}}$. 
The $\Gm$-action on our cone $N$ does not satisfy this condition.

\subsection{The general statement of the $\Gm$-action}
\begin{defn}
Let $(N,s)$ be a $d$-critical  scheme. 
A $\Gm$-action 
$$\mu: \Gm\times N\to N$$
is $\Gm$-invariant if $\mu(\gamma)^\star(s)=s$ for any $\gamma\in\Gm$, and 
$s\in H^0(\sS_{N})$, or equivalently 
$\mu^\star(s)=\pi_{N}^\star(s)$, where $\pi_N: \Gm\times N\to N$ is the projection. 
\end{defn}

\begin{defn}
The $\Gm$-action on a scheme $N$ is called {\em good} if any orbit is contained in an affine  subscheme of 
$N$. Equivalently, there exists an open cover of  subschemes which are $\Gm$-invariant. 
\end{defn}

Here are two results in \cite{Joyce}, and  \cite{Jiang3}. 

\begin{prop}
Let $(N,s)$ be a $d$-critical scheme which is $\Gm$-invariant under the $\Gm$-action. Then 
\begin{enumerate}
\item If the action $\mu$ is good, then any $x\in N$ there exists a $\Gm$-invariant critical chart $(R,U,f,i)$ on 
$(N,s)$ such that $\dim(T_{x}N)=\dim(U)$;
\item If for all $x\in N$ we have a $\Gm$-invariant critical chart $(R,U,f,i)$, then the action  $\mu$ is good. 
\end{enumerate}
\end{prop}

\begin{prop}\label{prop_Gm_fixed_d_critical_scheme}
Let $(N,s)$ be a $d$-critical scheme which is $\Gm$-invariant under the $\Gm$-action. Let $N^{\Gm}$ be the fixed subscheme.   Then the fixed subscheme
$N^{\Gm}$ inherits a formal $d$-critical scheme structure 
$(N^\Gm, s^{\Gm})$, where $s^\Gm=i^\star(s)$ and 
$i: N^\Gm\hookrightarrow N$ is the inclusion map. 
\end{prop}

Let $(N,s)$ be a $d$-critical $\kappa$-scheme with a good $\Gm$-action.  Let 
$$N^{\Gm}=\bigsqcup_{i\in J} N_i^{\Gm}$$
be the decomposition of the fixed locus $X^{\Gm}$ into connected components, such that 
$(N_i^{\Gm}, s_i^{\Gm})$ are oriented  $d$-critical schemes.   The action $\Gm$ has a decomposition 
$$T_x(N_i)=(T_x N_i)_{0}\oplus T_x(N_i)_{+}\oplus (T_x(N_i))_-$$
where the direct sums are the parts of zero, positive and negative weights with respect to the $\Gm$-action.  
Maulik \cite{Maulik} defined the virtual index
\begin{equation}\label{virtual_index}
\ind^{\virt}(N_i^{\Gm},N)=\dim_{\kappa}(T_{x}(N)_+)-\dim_{\kappa}(T_{x}(N)_-)
\end{equation}
so that it is constant on the strata $N_i^{\Gm}$.  This virtual index is essential in \cite{Maulik} to prove the motivic localization formula, see also \cite[Theorem 7.17]{Jiang3}.  We don't need it here. 

Suppose that the $d$-critical $\kappa$-scheme $(N,s)$ is oriented, i.e., there exists a square root 
$K_{N,s}^{\frac{1}{2}}$ for the canonical line bundle $K_{N,s}$.  Then we denote by $\mMF_{N,s}^{\phi}\in \overline{\mM}^{\hat{\mu}}_{N}$ the global motive for $N$ in \cite{BJM}. Here 
$\overline{\mM}^{\hat{\mu}}_{N}=K_0^{\hat{\mu}}(\Var_{N})[\ll^{-1}]$ is the equivariant Grothendieck ring of varieties over $X$ with a new product $\odot$ defined in \cite{BJM}, and reviewed in \cite[\S 7.3]{Jiang3}. 

\begin{defn}
We call the action 
$$\mu: \Gm\times N\to  N$$
{\em circle-compact} if the limit $\lim_{\lambda\to 0}\mu(\lambda)x$ exists for any $x\in N$. If $X$ is proper, then any $\Gm$-action on $N$ is circle-compact. 
\end{defn}

For the $d$-critical scheme $(N,s)$ in  Proposition \ref{prop_d_critical_scheme_N}, 
we have:
\begin{prop}
There exists a good and circle compact $\Gm$-action on the oriented $d$-critical scheme 
$(N,s)$ for the cone $\pi: N\to M$. 
\end{prop}

\subsection{The $\Gm$-localization formula for the cone $N$}

\begin{thm}\label{thm_general_motivic_localization}
Let $(N,s)$ be the $d$-critical scheme in Proposition \ref{prop_d_critical_scheme_N} and $\mu$ is the good, circle-compact action of  $\Gm$ on $N$. Assume that there is an orientation $K_{N,s}^{\frac{1}{2}}$ and let  
$\mMF_{N}^{\phi}\in\overline{\mM}_{N}^{\hat{\mu}}$ be the global motive of $N$.   Then we have the following localization formula. 
$$\int_{N}\mMF_{N}^{\phi}=\ll^{-(\rank(E^0)-\rank(E^{-1}))/2}\odot \int_{M}[M].$$ 
\end{thm}
\begin{rmk}
We remark here that $\int_{N}\mMF_{N}^{\phi}$ means pushforward to a point, hence the absolute motive in 
$\overline{\mM}_{\kappa}^{\hat{\mu}}$. 

In practice, our $N$ will be the moduli space of stable Higgs sheaves $(E,\phi)$ on a surface or surface DM stack $S$ as in \cite{TT}, \cite{JP}.  A recent result in \cite{JU} shows that there is a natural orientation data on the moduli stack of coherent sheaves on Calabi-Yau threefold, and this orientation data, when restricting to the moduli space of torsion 2-dimensional sheaves on $X:=\Tot(K_S)$ (which is isomorphic to $N$) gives the orientation  $K_{N,s}^{\frac{1}{2}}$.
\end{rmk}
\begin{proof}
Let $(\widetilde{R},  \widetilde{A}, \s, \widetilde{i})$ be a $d$-critical chart of $(N,s)$.  Recall that from \cite{BJM}, \cite{Jiang3}, the motive 
$$\mMF_{N}^{\phi}|_{\widetilde{R}}=i^\star(\mMF_{\widetilde{A},\s}^{\phi})\odot \Upsilon(Q_{\widetilde{R},  \widetilde{A}, \s, \widetilde{i}})\in \overline{\mM}_{\widetilde{R}}^{\hat{\mu}},$$
where 
$\mMF_{\widetilde{A},\s}^{\phi}=\ll^{-\dim(\widetilde{A})/2}\odot [[\widetilde{A}_0,\hat{\iota}]-\mMF_{\widetilde{A},\s}]|_{\widetilde{R}}$ is the motivic vanishing cycle; 
$\mMF_{\widetilde{A},\s}$ is the motivic nearby cycle defined in \cite[Definition 7.9]{Jiang3}; and 
$\Upsilon(Q_{\widetilde{R},  \widetilde{A}, \s, \widetilde{i}})=\ll^{-\frac{1}{2}}\odot ([\widetilde{R},\hat{\iota}]-[Q,\hat{\rho}])\in \overline{\mM}_{\widetilde{R}}^{\hat{\mu}}$ is the motive of the principal $\zz_2$-bundle 
$Q_{\widetilde{R},  \widetilde{A}, \s, \widetilde{i}}$ as in \cite[\S 2.5]{BJM} and recalled in \cite[\S 7.3]{Jiang3}. 
The canonical line bundle $K_{N,s}$ is the unique line bundle on $N$ such that 
on the $d$-critical chart $(\widetilde{R},  \widetilde{A}, \s, \widetilde{i})$, there is a natural isomorphism 
$$\iota: K_{N,s}|_{\widetilde{R}^{\red}}\stackrel{\cong}{\longrightarrow} 
\widetilde{i}^{\star}(K_{\widetilde{A}}^{\otimes 2})|_{\widetilde{R}^{\red}}.$$
The principal $\zz_2$-bundle 
$Q_{\widetilde{R},  \widetilde{A}, \s, \widetilde{i}}$ parametrizes the  local isomorphisms
\begin{equation}\label{eqn_local_principal_Z_2-bundle}
K_{N,s}^{1/2}|_{\widetilde{R}^{\red}}\stackrel{\cong}{\longrightarrow} 
\widetilde{i}^{\star}(K_{\widetilde{A}}^{})|_{\widetilde{R}^{\red}}
\end{equation}
given by the orientation. 
The product $\odot$ in  $\overline{\mM}_{\widetilde{R}}^{\hat{\mu}}$ is given in \cite[Definition 2.3]{BJM} and we reviewed it in 
Definition \ref{defn_odot_product}. 

We now prove the result on the $d$-critical chart  $(\widetilde{R},  \widetilde{A}, \s, \widetilde{i})$.  We prove that the following formula holds:
\begin{equation}\label{local_picture_formula}
\int_{\widetilde{R}}\mMF_{\widetilde{A},\s}^{\phi}=\ll^{-\frac{1}{2}\dim(\widetilde{A})+\rank(E)}\odot \int_{R}[R].
\end{equation}

 Recall  from Proposition \ref{prop_d_critical_scheme_N} that in local coordinates each $d$-critical chart of $N$ is given by 
$$(R, \Tot(E^*)|_{U}, \widetilde s, i),$$
where  Tot$\;(E^*)|_{U}$  is the bundle $E^*$ over an open neighborhood $U$ of $A$, $\s$ is the local regular function, and $R:=\Crit(\s)$, $i: R\hookrightarrow \Tot(E^*)|_{U}$ is the inclusion. 
Trivialising $E$ with a basis of sections $e_j$, we get a dual basis $f_j$ for $E^*$ and coordinates $y_j$ on the fibres of Tot$\;(E^*)$.  The function 
$$\widetilde s=\sum_j s_jy_j,$$
where we  write the section  $s=\sum_js_je_j,\ \tau=\sum_jy_jf_j$.  The group  $\Gm$ acts  on $N$ by scaling the fibre and the fixed locus is $M\subset N$.  So the invariant part $\s^{\Gm}=0$.

Since locally $\widetilde{A}=A\times \aaa_{\kappa}^{\rank(E^*)}$, and the $\Gm$-action on $N$ actually extends to an action on 
$\widetilde{A}$.  
Here we assume that the 
group 
$\Gm$ acts on $A$ trivially, and on $\aaa_{\kappa}^{\rank(E^*)}$ by positive weights $w_1>0, \cdots, w_r>0$.  
Let $\Gm$ act on $\aaa_{\kappa}^1$ by weight $d>0$. 
Our cone $N\to M$ case is exactly the special case that all $w_i, d$ are equal to one. 
So this is exactly the situation as in Theorem 4.1.1. in \cite{NS}.  So from \cite[Theorem 4.1.1.]{NS} we have the motivic nearby cycle 
$$\mMF_{\widetilde{A},\s}=[\s^{-1}(1)].$$

From the proof of  Proposition \ref{prop_d_critical_scheme_N}, the critical locus $\widetilde{R}\subset \widetilde{A}$ has ideal 
$(s_i, Ds_i)$.  So over $\widetilde{R}\setminus R$, i.e., the part of $\widetilde{R}$ such that 
$s_i\neq 0$, $\s^{-1}(0)$ and $\s^{-1}(1)$ are all afine subspaces in $\widetilde{A}$. Hence 
 $[\s^{-0}(0)-\s^{-1}(1)]|_{\widetilde{R}\setminus R}=0$.  
We calculate
\begin{align*}
\mMF_{\widetilde{A},\s}^{\phi}|_{\widetilde{R}}&=\ll^{-\dim(\widetilde{A})/2}\odot [[\widetilde{A}_0,\hat{\iota}]-\mMF_{\widetilde{A},\s}]|_{\widetilde{R}}\\
&=\ll^{-\dim(\widetilde{A})/2}\odot[\s^{-0}(0)-\s^{-1}(1)]|_{\widetilde{R}} \\
&=\ll^{-\dim(\widetilde{A})/2}\odot\left([\s^{-0}(0)-\s^{-1}(1)]|_{R}+ [\s^{-0}(0)-\s^{-1}(1)]|_{\widetilde{R}\setminus R}\right) \\
&=\ll^{-\dim(\widetilde{A})/2}\odot [\s^{-1}(0)]|_{R}\\ 
&=\ll^{-\dim(\widetilde{A})/2}\odot \ll^{\rank(E^*)}\cdot [R].
\end{align*}

Let us argue the invariant 
$$\Upsilon(Q_{\widetilde{R},  \widetilde{A}, \s, \widetilde{i}})=\ll^{-\frac{1}{2}}\odot ([\widetilde{R},\hat{\iota}]-[Q,\hat{\rho}])\in \overline{\mM}_{\widetilde{R}}^{\hat{\mu}}$$
on the   $d$-critical chart $(\widetilde{R},  \widetilde{A}, \s, \widetilde{i})$. 
Shrink $\widetilde{R}$ if necessary, the orientation
$K_{N,s}^{1/2}$ is trivial on $\widetilde{R}$.  Then the principal $\zz_2$-bundle 
$Q_{\widetilde{R},  \widetilde{A}, \s, \widetilde{i}}$ is trivial  parametrizing the isomorphisms in (\ref{eqn_local_principal_Z_2-bundle}), so 
\begin{align*}
\Upsilon(Q_{\widetilde{R},  \widetilde{A}, \s, \widetilde{i}})&=\ll^{-\frac{1}{2}}\odot ([\widetilde{R},\hat{\iota}]-[Q,\hat{\rho}])\in \overline{\mM}_{\widetilde{R}}^{\hat{\mu}}\\
&=\ll^{-\frac{1}{2}}\odot \ll^{\frac{1}{2}}\odot [\widetilde{R},\hat{\iota}]\\
&=[\widetilde{R},\hat{\iota}].
\end{align*}
Since the motive $[\widetilde{R},\hat{\iota}]$ is the identity in $\overline{\mM}_{\widetilde{R}}^{\hat{\mu}}$, we have 
$$\mMF_{N}^{\phi}|_{\widetilde{R}}=\ll^{-\dim(\widetilde{A})/2}\cdot \ll^{\rank(E^*)}\cdot [R].$$
This implies that the global motive $\mMF_{N,s}^{\phi}$ is locally on the $d$-critical charts given by the absolute motive of the base $[R]$ in $M$,
so the motives $[R]$ glue to give $[M]$. 
Finally 
$$-\dim(\widetilde{A})/2+\rank(E^*)=-\frac{1}{2}\dim(A)+\frac{1}{2}\rank(E^*)=-\frac{1}{2}\vd.$$
The result follows. 
\end{proof}

\begin{rmk}
Theorem  \ref{thm_general_motivic_localization} implies one of the main result 
$$\chi(N, \nu_{N})=(-1)^{\vd}\chi(M)$$
in \cite{JT}, which has applications in the study of Vafa-Witten invariants \cite{TT}, \cite{TT2}. 

One easily see this by taking the limit $\lim_{\ll^{\frac{1}{2}}\to -1}$, and note that 
$$\lim_{\ll^{\frac{1}{2}}\to -1}\int_{N}\mMF_{N}^{\phi}=\chi(N, \nu_{N})$$
and 
$(-1)^{-\rank(E^0)-\rank(E^{-1})}=(-1)^{\dim \Tot(E^*)}=(-1)^{\vd}$, we have 
$$\chi(N, \nu_{N})=(-1)^{\vd}\chi(M).$$ 
\end{rmk}

\section{Motivic Vafa-Witten theory}\label{sec_motivic_VW}

In this section we review the motivic Vafa-Witten invariants and give calculations for K3 surfaces.

\subsection{Power structure on motivic ring}\label{subsec_power_structure}

We recall the power structure on the Grothendieck ring $K_0(\Var_{\kappa})$ following \cite{Gusein-Zade}. 
Let $R$ be a ring. A power structure on $R$ is a map
$$(1+q R[\![q]\!])\times R\to 1+q R[\![q]\!])$$
by
$$(A(t), m)\mapsto (A(q))^m$$
which satisfies the properties: $A(q)^0=1$, $A(q)^1=A(q)$, $A(q)^{m+n}=A(q)^{m} A(q)^n$, $A(q)^{mn}=(A(q)^{m})^{n}$,
$A(q)^m B(q)^m=(A(q) B(q))^m$, and $(1+q)^m=1+mq+O(q^2)$. 
From  \cite{Gusein-Zade}, there exists a power structure on the Grothendieck ring $K_0(\Var_{\cc})$ , defining uniquely by the property 
that for a variety $X$, 
$$(1-q)^{-[X]}=\sum_{k=0}^{\infty}[\Sym^k(X)]q^k$$
is the motivic zeta function. 

Also from  \cite{Gusein-Zade},  we have the exponential  map
\begin{equation}\label{eqn_exponential_map}
\Exp:  q \overline{\mM}_{\cc}[\![q]\!]\to 1+q \overline{\mM}_{\cc}[\![q]\!]
\end{equation}
by
$$\Exp\left(\sum_{n=1}^{\infty}[A_n]q^n\right)=\prod_{n\geq 1}(1-q^n)^{-[A_n]}.$$
This equation satisfies the substitution rule:
\begin{equation}\label{eqn_substitution_rule}
\Exp (A(q))|_{q\mapsto (-\ll^{\frac{1}{2}})^n q}  =\Exp\left(A\left( (-\ll^{\frac{1}{2}})^n q\right)\right).
\end{equation}

\subsection{Motivic Vafa-Witten invariants}\label{subsec_motivic_VW}

In this section we  define the motivic Vafa-Witten invariants. Let us fix a smooth surface $S$ and $X:=\Tot(K_S)$ be the total space of canonical line bundle $K_S$.  In \cite{TT}, Tanaka-Thomas defined the moduli space of stable Higgs pairs $(E,\phi)$, where $E$ is a torsion free rank $r$ coherent sheaf on $S$, and $\phi: E\to E\otimes K_S$ is a $\oO_S$-linear map called the Higgs field.  The detail definition of the Gieseker stability and the construction of the moduli spaces are in \cite{TT}.  We work on the moduli space 
$$N_L^{\perp}:=N_L^{\perp}(S, c)$$
of stable Higgs pairs on $S$ with topological invariant $c=(r, c_1, c_2)$ and 
$\tr(\phi)=0$, $\det(E)=L\in \Pic(S)$.  From spectral theory in \cite[\S 2]{TT}, every Higgs pair $(E,\phi)$ corresponds to a two dimensional torsion sheaf $\eE_{\phi}$ supported on $S\subset X$. The Gieseker stability of the Higgs pair $(E,\phi)$ transforms to the Gieseker stability of the  torsion sheaf $\eE_{\phi}$, and the moduli space $N_L^{\perp}$ is isomorphic to the moduli space of stable two dimensional torsion sheaves on $X$ with fixed determinant.  Thus the moduli space $N_L^{\perp}$ admits a symmetric obstruction theory as in \cite{Behrend}.  $N_L^{\perp}$ is not compact, but admits a $\cc^*$-action scaling the Higgs field (or induced by the scaling fiber action on $X$), the $\cc^*$-fixed locus contains two branches.

\textbf{Instanton Branch}:  The first component contains  Higgs pairs $(E,\phi)$ which is $\cc^*$-fixed and $\phi=0$.  This component is just the moduli space $M:=M_L(S,c)$ of stable coherent sheaves on $S$ with topological invariant $c$. Note that in this case $E$ is stable. 

\textbf{Monopole Branch}:  The second component contains  Higgs pairs $(E,\phi)$ which is $\cc^*$-fixed and $\phi\neq 0$.
This component usually contains pair $(E,\phi)$, where $E$ splits to direct sum of subsheaves which maybe nonsemistable. 
For general type surfaces $S$ with $p_g>0$, some component  in this loci is isomorphic to nested Hilbert schemes on $S$, see \cite{TT}, \cite{Thomas}. 

The $\SU(r)$-Vafa-Witten invariants are defined by virtual localization \cite{GP} by:
\begin{equation}\label{eqn_VW_TT}
\VW_c(S):=\int_{[(N_L^{\perp})^{\cc^*}]^{\virt}}\frac{1}{e(N^{\virt})}
\end{equation}
where $e(N^{\virt})$ is the Euler class of the virtual normal bundle.  The small $\SU(r)$-Vafa-Witten invariants are defined by:
\begin{equation}\label{eqn_vw_TT}
\vw_c(S):=\chi(N_L^{\perp}, \nu_{N})
\end{equation}
which is the weighted Euler characteristic of $N_L^{\perp}$ weighted by the Behrend function $\nu_{N}: N_L^{\perp}\to \zz$. 

In general, $\VW_c(S)\neq \vw_c(S)$, but they are equal in two cases: the case $K_S<0$, and the case of K3 surfaces \cite{MThomas}.

Since $N_L^{\perp}$ is the moduli space of stable sheaves on $X$, from \cite{JU}, this moduli space admits an orientation 
$K_{N_L^{\perp}}^{\frac{1}{2}}$.  Therefore from \cite{BJM}, there exists a global motive 
$\mMF^{\phi}_{N_L^{\perp}}\in \overline{\mM}_{N_L^{\perp}}^{\hat{\mu}}$. 

\begin{defn}\label{defn_motivic_VW}
The motivic Vafa-Witten invariant $\vw_c^{m}(S)$ is defined by:
$$\vw_c^{M}(S):=\int_{N_L^{\perp}}[\mMF^{\phi}_{N_L^{\perp}}],$$
\end{defn}
where the integral means pushing forward to a point. 

\begin{rmk}
The motivic Vafa-Witten invariants can only be defined as a refinement for the invariants $\vw(S)$.  In \cite{Thomas}, R. Thomas defined the refinement K-theoretical Vafa-Witten invariants using virtual structure sheaf and orientation $K_{N_L^{\perp}}^{\frac{1}{2}}$, which refined the Vafa-Witten invariants $\VW(S)$. 
\end{rmk}

In order to apply Theorem \ref{thm_general_motivic_localization}, let 
$$N:=\left(N_L^{\perp}\right)^{\text{stable}}\subset N_L^{\perp}$$
be the locus in $N_L^{\perp}$ consisting of Higgs pairs $(E,\phi)$ such that $E$ is stable. Then $N\to M$ is a cone over the moduli space of stable sheaves on $S$, and is the situation in \S \ref{constr}.  $N$ is a quasi-projective scheme and also admits an induced symmetric obstruction theory from $N_L^{\perp}$. The restriction of the orientation $K_{N_L^{\perp}}^{\frac{1}{2}}$ gives an orientation 
$$K_{N}^{\frac{1}{2}}:=K_{N_L^{\perp}}^{\frac{1}{2}}|_{N}.$$
The global motive $\mMF^{\phi}_{N_L^{\perp}}$ restricts 
to give the global motive 
$$\mMF^{\phi}_{N}=\mMF^{\phi}_{N_L^{\perp}}|_{N}.$$
The $\cc^*$-fixed locus of $N$ is $M=M_L(S,c)$.  We review the splitting of the perfect obstruction theory of $N$ on $M$ a bit here following \cite[\S 3]{TT}.
Let 
$$p_X: N\times X\to X; \quad  p_S: N\times S\to S $$
be the projections and $\rE$ on $N\times X$ the universal sheaf  (or Higgs pair). 
Let $\pi: N\times X\to N\times S$ be the projection. Then the pushforward $\pi_*\rE=\E$ is the universal sheaf on $N\times S$. 
Then 
from the diagram in  
(\cite[Corollary 2.22]{TT}):
\[
\xymatrix{
R\cHom_{p_{S}}(\E, \E\otimes K_{S})_{0}[-1]\ar[r]\ar@{<->}[d] &R\cHom_{p_{X}}(\rE, \rE)_{\perp}
\ar[r]\ar@{<->}[d] & R\cHom_{p_{S}}(\E, \E)_{0}\ar@{<->}[d]\\
R\cHom_{p_{S}}(\E, \E\otimes K_{S})[-1]\ar[r]\ar@{<->}[d]_{\id}^{\tr} &R\cHom_{p_{X}}(\rE, \rE)
\ar[r]\ar@{<->}[d] & R\cHom_{p_{S}}(\E, \E)\ar@{<->}[d]_{\id}^{\tr}\\
Rp_{S *}K_{S}[-1]\ar@{<->}[r]& Rp_{S *}K_{S}[-1]\oplus Rp_{S *}\oO_{S}\ar@{<->}[r]& Rp_{S *}\oO_{S}
}
\]
where $(-)_0$ denotes the trace-free Homs.  The $R\cHom_{p_{X}}(\rE, \rE)_{\perp}$ is the co-cone of the middle column and it will provide the symmetric obstruction theory of the moduli space $N_{L}^{\perp}$ of stable trace free fixed determinant Higgs pairs. 
Thus we have:
$$E_N^{\bullet}:=R\cHom_{p_{X}}(\rE, \rE)_{\perp}[1]\mathfrak{t}^{-1}\cong 
R\cHom_{p_{S}}(\E, \E\otimes K_{S})_0[1]\oplus  R\cHom_{p_{S}}(\E, \E)_0[2]\mathfrak{t}^{-1}$$
where $\mathfrak{t}^{-1}$ represents the moving part of the $\cc^*$-action, which is the standard representation of $\cc^*$. 
Then 
$$E_M^{\bullet}:=R\cHom_{p_{S}}(\E, \E\otimes K_{S})_0[1]\to \ll_{M}$$
gives the perfect obstruction theory on $M$. 
Let $\vd:=\rank(E_M^{\bullet})$ be the virtual dimension. Then from Theorem \ref{thm_general_motivic_localization}, 
\begin{equation}\label{eqn_integral_M}
\int_{N}\mMF_N^{\phi}=\ll^{-\vd/2}\int_{M}[M].
\end{equation}

\begin{rmk}
Taking Euler characteristic of (\ref{eqn_integral_M}) gives:
$$\chi(N,\nu_N|_{N})=(-1)^{\vd}\chi(M)$$
which is the invariants $\vw(S)$ contributed from $N\subset N_L^{\perp}$. 
Formula  (\ref{eqn_integral_M}) is useful for calculations for the motivic Vafa-Witten invariants when $N_L^{\perp}$ has no 
monopole branch fixed loci. 
\end{rmk}

\subsection{Proof of Theorem \ref{thm_motivic_series_vw_K3}-Calculations for K3 surfaces}

Let $S$ be a smooth projective K3 surface.  In this section we perform a calculation for the motivic invariants for $S$. 
In this case $X=S\times \cc$, and any $\cc^*$-fixed Higgs pair $(E,\phi)$ has vanishing Higgs field $\phi=0$. Therefore the moduli space $N=N_L^{\perp}\to M$ is a cone over the moduli space of stable sheaves $M$.  We only have instanton branch. 

Fix a topological invariant $c_0=(r, c_1, c_2)\in H^*(S,\zz)$ such that $(r, c_1)=1$. Then semistablity coincides with stability and  the moduli space $M=M_L(S,c_0)$ is an irreducible symplectic variety which is birational equivalent to the Hilbert scheme 
$\Hilb^d(S)$ of point on $S$, see \cite{Yoshioka}.  Here 
$$d=1-\chi_S(c_0, c_0).$$
Sine the moduli space $M=M_L(S,c_0)$ is smooth, its virtual dimension is the dimension 
$$\vd=2-\chi_S(c_0,c_0).$$ 
Thus 
$$\int_{N}[\mMF_N^{\phi}]=\ll^{\frac{1}{2}\chi_S(c_0,c_0)-1}\int_M[M].$$
We calculated 
$$d=1+r(c_2-r)+\frac{1-r}{2}c^2_1.$$
Thus we write down the motivic generating series as:
\begin{equation}\label{eqn_motivic_series_vw_1}
\vw_{r, c_1}^M(S)=\sum_{c_2}[\mMF_N^{\phi}]q^{c_2}=\sum_{c_2}\ll^{\frac{1}{2}\chi_S(c_0,c_0)-1}[\Hilb^d(S)]q^{c_2},
\end{equation}

We calculate the motivic series (write $c_2$ as integers $m\in\zz$):
\begin{align*}
&\vw_{r, c_1}^M(S)\\
&=\sum_{m\in \zz}\ll^{-(1+r(m-r)+\frac{1-r}{2}c^2_1)}[\Hilb^{1+r(m-r)+\frac{1-r}{2}c^2_1}(S)]q^{m}\\
&=\sum_{m\in \zz}\ll^{-(1+r(m-r)+\frac{1-r}{2}c^2_1)}[\Hilb^{1+r(m-r)+\frac{1-r}{2}c^2_1}(S)]
\cdot \left(q^{\frac{1}{r}}\right)^{1+r(m-r)+\frac{1-r}{2}c^2_1}\left(q\right)^{r-\frac{1}{r}-\frac{1-r}{2r}c_1^2}\\
&=q^{r-\frac{1}{r}-\frac{1-r}{2r}c_1^2}\cdot \sum_{m\in \zz}[\Hilb^{1+r(m-r)+\frac{1-r}{2}c^2_1}(S)]\left(\ll^{-1}q^{\frac{1}{r}}\right)^{1+r(m-r)+\frac{1-r}{2}c^2_1}\\
\end{align*}
Now using the exponential form, see \cite{BBS}, \cite{Gottsche}, \cite{Gusein-Zade},
$$\sum_{m\in \zz}[\Hilb^m(S)]q^m=\Exp\left(\frac{[S]q}{1-\ll q}\right)$$
and therefore 
$$\sum_{m\in \zz}\ll^m[\Hilb^m(S)]q^m=\Exp\left(\frac{[S]\cdot \ll q}{1-\ll^2 q}\right).$$
Then we have
\begin{align*}
\vw_{r, c_1}^M(S)&=q^{r-\frac{1}{r}-\frac{1-r}{2r}c_1^2}\cdot \sum_{m\in \zz}[\Hilb^{1+r(m-r)+\frac{1-r}{2}c^2_1}(S)]\left(\ll^{-1}q^{\frac{1}{r}}\right)^{1+r(m-r)+\frac{1-r}{2}c^2_1}\\
&=q^{r-\frac{1}{r}-\frac{1-r}{2r}c_1^2}\cdot\frac{1}{r}\sum_{j=0}^{r-1}e^{\pi i \frac{r-1}{r}j c_1^2}\cdot \Exp\left(\frac{[S]\cdot \ll^{-1} e^{\frac{2\pi i j}{r}}q^{\frac{1}{r}}}{1-e^{\frac{2\pi i j}{r}}q^{\frac{1}{r}}}\right).
\end{align*}
We also can write it as the infinite product form. From (\ref{eqn_exponential_map}),  and since 
\begin{align*}
\frac{[S]\cdot \ll^{-1} e^{\frac{2\pi i j}{r}}q^{\frac{1}{r}}}{1-e^{\frac{2\pi i j}{r}}q^{\frac{1}{r}}}&=\sum_{n=0}^{\infty}[S]\cdot \ll^{-1} e^{\frac{2\pi i j}{r}}q^{\frac{1}{r}}\cdot (e^{\frac{2\pi i j}{r}}q^{\frac{1}{r}})^{n+1}\cdot (e^{\frac{2\pi i j}{r}}q^{\frac{1}{r}})^{-1}\\
&=\sum_{n=1}^{\infty}\ll^{-1}[S]\cdot (e^{\frac{2\pi i j}{r}}q^{\frac{1}{r}})^{n}.
\end{align*}
Therefore 
$$\Exp\left(\frac{[S]\cdot \ll^{-1} e^{\frac{2\pi i j}{r}}q^{\frac{1}{r}}}{1-e^{\frac{2\pi i j}{r}}q^{\frac{1}{r}}}\right)=
\prod_{n\geq 1}\left(1-\left(e^{\frac{2\pi i j}{r}}q^{\frac{1}{r}}\right)^n\right)^{-\ll^{-1}[S]}.
$$
Then 
$$\vw_{r, c_1}^M(S)=q^{r-\frac{1}{r}-\frac{1-r}{2r}c_1^2}\cdot\frac{1}{r}\sum_{j=0}^{r-1}e^{\pi i \frac{r-1}{r}j c_1^2}\cdot\prod_{n\geq 1}\left(1-\left(e^{\frac{2\pi i j}{r}}q^{\frac{1}{r}}\right)^n\right)^{-\ll^{-1}[S]}.$$
$\square$

\subsection{Proof of the $\chi_y$-genus in Corollary \ref{cor_chi_y_series_vw_K3}}
From \cite{Gusein-Zade},  there exists a ring homomorphism 
$$e: K_0(\Var_{\cc})\to \zz[u,v]$$
given by sending each variety $X$ to its Hodge-Deligne polynomial
$$X\mapsto e_X(u,v)=\sum_{i,j}h^{ij}(X)(-u)^i (-v)^j.$$
The $\chi_y$-genus is just $e_X(y, 1)$.  
The $\chi_y$-polynomial of the K3 surface $S$ is $\chi_y(S)=2+20 y +2 y^2$.
For the Hilbert scheme $\Hilb^m(S)$, from the properties the ring homomorphism
in \cite[Theorem 2]{Gusein-Zade} 
we have 
\begin{align*}
&\chi_y\left(\Exp\left(\frac{[S]\cdot \ll^{-1} e^{\frac{2\pi i j}{r}}q^{\frac{1}{r}}}{1-e^{\frac{2\pi i j}{r}}q^{\frac{1}{r}}}\right)\right)\\
&=\chi_y\left(\prod_{n\geq 1}\left(1-\left(e^{\frac{2\pi i j}{r}}q^{\frac{1}{r}}\right)^n\right)^{-\ll^{-1}[S]}\right)\\
&=\prod_{n\geq 1}\left(1-\left(e^{\frac{2\pi i j}{r}}q^{\frac{1}{r}}\right)^n\right)^{-y^{-1}(2+20 y+2 y^2)}\\
&=\prod_{n\geq 1}\left(1-\left(e^{\frac{2\pi i j}{r}}q^{\frac{1}{r}}\right)^n\right)^{-20}\left(1-y\left(e^{\frac{2\pi i j}{r}}q^{\frac{1}{r}}\right)^n\right)^{-2}\left(1-y^{-1}\left(e^{\frac{2\pi i j}{r}}q^{\frac{1}{r}}\right)^n\right)^{-2}.
\end{align*}
We define 
$$\widetilde{\Delta}(q, y):=\prod_{n\geq 1}\left(1-q^n\right)^{20}\left(1-y q^n\right)^{2}\left(1-y^{-1}q^n\right)^{2}$$
Then we get the $\chi_y$-genus of the Vafa-Witten invariants:
$$\vw_{r, c_1}^{\chi_y}(S)=q^{r-\frac{1}{r}-\frac{1-r}{2r}c_1^2}\cdot \frac{1}{r}\sum_{j=0}^{r-1}e^{\pi i \frac{r-1}{r}j c_1^2}\cdot\widetilde{\Delta}(e^{\frac{2\pi i j}{r}}q^{\frac{1}{r}}, y)^{-1}.$$

Note that this result is similar to \cite[Theorem 5.48]{Thomas}, where Thomas got similar results  using K-theoretical Vafa-Witten invariants.

\bibliographystyle{halphanum}
\bibliography{references}

\bigskip \noindent \\
{\tt y.jiang@ku.edu} \medskip

\noindent Department of Mathematics \\
\noindent University of Kansas \\
\noindent 405 Snow Hall, 1460 Jayhawk Blvd \\
\noindent Lawrence, KS 66045. USA

\end{document}